\documentclass[12pt]{article}
\usepackage{amssymb,url}
\newtheorem{theorem}{Theorem}[section]
\newtheorem{prop}[theorem]{Proposition}
\newtheorem{cor}[theorem]{Corollary}
\newtheorem{unnumber}{}

\newtheorem{prepf}[unnumber]{Proof}
\newenvironment{proof}{\prepf\rm}{\endprepf}
\newcommand{\qed}{\hfill$\Box$}
\newtheorem{preex}{Example}
\newenvironment{example}{\preex\rm}{\endpreex}

\newcommand{\ComGr}{\mathop{\mathrm{Com}}}
\newcommand{\DComGr}{\mathop{\mathrm{DCom}}}
\newcommand{\EPowGr}{\mathop{\mathrm{EPow}}}

\begin{document}

\title{Between the enhanced power graph and the commuting graph}
\author{Peter J. Cameron\footnote{School of Mathematics and Statistics,
University of St Andrews, North Haugh, St Andrews, Fife KY16 9SS, U.K.;
email \texttt{pjc20@st-andrews.ac.uk}}\quad
and Bojan  Kuzma\footnote{University of Primorska, Glagolja\v ska 8, Koper,
Slovenia; email\hfil\break \texttt{bojan.kuzma@famnit.upr.si}} 
\thanks{The second author acknowledges the financial support from the Slovenian Research
Agency, ARRS (research core funding No. P1-0222 and No. P1-0285).}}

\date{}
\maketitle

\begin{abstract}
The purpose of this note is to define a graph whose vertex set is a finite
group $G$, whose edge set is contained in that of the commuting graph of $G$
and contains the enhanced power graph of $G$. We call this graph the deep
commuting graph of $G$. Two elements of $G$ are joined in the deep commuting
graph if and only if their inverse images in every central extension of $G$
commute.

We give conditions for the graph to be equal to either of the enhanced power
graph and the commuting graph, and show that the automorphism group of $G$
acts as automorphisms of the deep commuting graph.
\end{abstract}

\section{Introduction}

Let $G$ be a finite group. Among a number of graphs defined on the vertex set
$G$ reflecting some algebraic properties of the group, two which have been
studied are the following:
\begin{itemize}\itemsep0pt
\item the \emph{commuting graph}, $\ComGr(G)$, first defined by Brauer and
Fowler~\cite{bf},
in which two elements $x$ and $y$ are joined if they commute;
\item the \emph{enhanced power graph}, $\EPowGr(G)$, first introduced by 
Aalipour \emph{et al.}~\cite{aetal}, in which $x$ and $y$ are joined if they
generate a cyclic subgroup of $G$.
\end{itemize}

It is clear that the edge set of the enhanced power graph is contained in that
of the commuting graph. (That is, $\EPowGr(G)$ is a \emph{spanning subgraph}
of $\ComGr(G)$.) The purpose of this note is to define a graph whose
edge set is between these two (contained in the first and containing the
second). We will call it the \emph{deep commuting graph}, for reasons which
will hopefully become clear, and denote it by $\DComGr(G)$.

\paragraph{Remark}
These are not the usual symbols used for these graphs in the literature, but
since we have several different graphs in play, we hope that these notations
will help reduce confusion. This graph will refine the hierarchy introduced
in \cite{cfrd}, which also includes \emph{power graph} of $G$~\cite{kq}, a
spanning subgraph of the enhanced power graph, and the \emph{non-generating
graph}~\cite{gk}, which contains the commuting graph if $G$ is non-abelian;
these graphs will not be considered here. Note also that the vertex set of our
commuting graph is $G$ while in some literature it is $G\backslash Z(G)$.

\medskip

The definition requires some preliminary discussion of Schur covers and
isoclinism.

\section{Covers and isoclinism}

In a group $G$, the \emph{centre} $Z(G)$ is the subgroup consisting of elements
which commute with every element of $G$. The \emph{derived group} $G'$ is the
subgroup generated by all \emph{commutators} $[x,y]=x^{-1}y^{-1}xy$ for
$x,y\in G$.

A \emph{central extension} of a finite group $G$ is a finite group $H$ with a
subgroup $Z\leqslant Z(H)$ such that $H/Z\cong G$. We will think of $G$ as a
quotient of $H$, in other words, the image of a homomorphism whose kernel
is $Z$. If $Z\leqslant Z(H)\cap H'$, it is called a \emph{stem extension}.
The subgroup $Z$ is the \emph{kernel} of the extension.

The deep commuting graph can be defined to have vertex set $G$, with an edge
joining  $x$ and $y$ if and only if $x\neq y$ and the preimages of $x$ and $y$
commute in every central extension of $G$. (It is sometimes more natural to
allow loops, but it makes no difference here.) However, we will give a more
explicit definition shortly.

As a temporary notation, we define the \emph{relative commuting graph} of
$G$ by $Z$ as the graph with vertex set $G$, in which  $x$ and $y$ are joined
if their inverse images in $H$ commute.

Schur showed that every finite group has a unique \emph{Schur multiplier}
$M(G)$, the central subgroup in the largest possible stem extension of $G$.
There are several other characterisations of $M(G)$: for example, it is the
first homology group $H_1(G,\mathbb{Z})$, or the first cohomology group
$H^1(G,\mathbb{C}^\times)$, or the quotient $(R\cap F')/[R,F]$ where $G$
has a presentation $F/R$ as quotient of a free group. The corresponding
stem extension $H$ is called a \emph{Schur cover} of $G$; it is not always uniquely
defined by $G$. For example, the dihedral and quaternion groups of order $8$
are both Schur covers of the Klein group of order $4$.

However, Jones and Wiegold~\cite{jw} showed that a relation weaker than
uniqueness holds, namely that of isoclinism. This means that the commutation
maps $\gamma$ from $H/Z(H)\times H/Z(H)$ to $H'$ defined by
$\gamma(Z(H)x,Z(H)y)=[x,y]$ are essentially the same in the two groups. (Note
that this is independent of the choice of coset representatives.) More
formally, two groups $H_1$ and $H_2$ are said to be \emph{isoclinic} if there
exist isomorphisms $\phi\colon H_1/Z(H_1)\to H_2/Z(H_2)$  and
$\psi\colon H_1'\to H_2'$ such that, for all $x,y\in H_1$,
\begin{equation}\label{eq:isoclinism}
\psi(\gamma_1(Z(H_1)x,Z(H_1)y))=\gamma_2(\phi(Z(H_1)x),\phi(Z(H_1)y)),
\end{equation}
where $\gamma_1$ and $\gamma_2$ are the commutation maps associated with $H_1$
and $H_2$; in other words, the isomorphisms respect the commutation map.
The pair $(\phi,\psi)$ is an \emph{isoclinism} from $H_1$ to $H_2$.

For example, if $H$ is any group and $A$ an abelian group, then $H$ and 
$H\times A$ are isoclinic.

Now an isoclinism maps the set of element-wise commuting pairs of cosets of
$Z(H_1)$ in $H_1$ to the corresponding set in $H_2$. So we have, as noted
in~\cite{pn}, that if $H_1$ and $H_2$ are isoclinic groups of the same order,
then their commuting graphs are isomorphic. In particular, Schur covers of a
group $G$ have isomorphic commuting graphs. More generally, if two groups
$G_1$ and $G_2$ are isoclinic, then $\ComGr(G_i)$ is the lexicographic
product of a complete graph on $Z(G_i)$ with a graph $\Gamma$ which is
the same for both groups.

It is clear that in (\ref{eq:isoclinism}) we may replace $Z(H_i)$ with any subgroup $Z_i$ in the center of $H_i$ provided that the groups $H_1/Z_1$ and $H_2/Z_2$ are isomorphic. Thus, if two central extensions of the same group $G$ are isoclinic (c.f. \cite[Definition  III.1.1]{bt}), then the induced commuting graphs on transversal of $Z_1$ in $H_1$ and on transversal of $Z_2$ in $H_2$ are isomorphic.

We note also that any central extension is isoclinic to a stem extension
\cite[Proposition III.2.6]{bt}, so we lose nothing by considering stem
extensions in what follows.

\paragraph{Definition} The \emph{deep commuting graph} of $G$, denoted by
$\DComGr(G)$, is the relative commuting graph of $G$ with respect to its Schur
multiplier $M(G)$. We note that in this definition the drawing of edges might
depend on choosing  a Schur cover, but the resulting graphs are all isomorphic.
We will show shortly that the edge set is in fact independent of the chosen
Schur cover.  The name is  meant to suggest that $x$ and $y$ are joined if the 
commutator of their preimages in a stem extension of $G$ lies as deep as
possible in the extension.

\begin{example}
As noted, the Klein group of order $4$ has two Schur covers up to isomorphism,
the dihedral and quaternion groups of order~$8$. In fact there are three
different covers isomorphic to the dihedral group, since any one of the three
involutions in the Klein group may be the one that lifts to an element of
order~$4$ in the cover, but a unique cover isomorphic to the quaternion group.
However, in all cases, the commuting graph of the cover is the lexicographic
product of a complete graph of order $2$ with a star $K_{1,3}$; so the deep
commuting graph of the Klein group is isomorphic to $K_{1,3}$ (which is in
fact equal to the enhanced power graph of this group).
\end{example}

\paragraph{Remark}
The relative commuting graph defines a map from quotients of the Schur
multiplier to spanning subgraphs of the commuting graph on $G$; this map is
order-reversing (but as we will see, not one-to-one).

\section{Commuting graph and commuting probability}

As a preliminary to this section, we note a connection between the commuting
graph and the commuting probability of a group $G$.

Let $G$ be a finite group. Then
$G$ acts on itself by conjugation; the stabiliser of a point $x$ is its
centraliser (the set of neighbours of $x$ in the commuting graph), while
the orbits are conjugacy classes. So the Orbit-counting Lemma shows that the
number of conjugacy classes is equal to the average valency of the commuting
graph. Dividing through by $|G|$, we see that the proportion of all ordered
pairs which commute is equal to the ratio of the number of conjugacy classes
to $|G|$. This fraction is called the \emph{commuting probability} of $G$, see
\cite{Eberhard}. We denote it by $\kappa(G)$, and note that it is the ordered
edge-density of the commuting graph (with a loop at each vertex).

\begin{prop}
Suppose that $\pi_i:H_i\to G$ are epimorphisms corresponding to central
extensions of $G$ for $i=1,2$. Suppose that there exists an epimorphism
$\phi:H_1\to H_2$ such that $\pi_2\circ\phi=\pi_1$. Then 
$\kappa(H_1)\leqslant\kappa(H_2)$.
\end{prop}

\begin{proof}
This follows from the fact that, for $x,y\in H_1$, if $x$ and $y$ commute then 
$\phi(x)$ and $\phi(y)$ commute in $H_2$. \qed
\end{proof}

Though it is not necessary for our argument, we note that in the cited paper
Eberhard proved that values of the commuting probability for finite groups
are well-ordered by the reverse of the usual order on the unit interval.
Thus, in the situation of the above proposition, either
$\kappa(H_1)=\kappa(H_2)$, or $\kappa(H_1)\leqslant\kappa(H_2)-\epsilon$,
where $\epsilon$ depends only on $\kappa(H_2)$.

\section{Properties of the deep commuting graph}

The two theorems of this section show that the edge set of the deep commuting
graph is independent of the chosen Schur cover, and that it lies between the
edge sets of the enhanced power graph and the commuting graph.

\begin{theorem}
Let $G$ be a finite group.
\begin{enumerate}
\item If $H_1$ and $H_2$ are Schur covers of $G$, then $H_1$ and $H_2$
define the same edge sets of the deep commuting graph (in other words, 
for $x,y\in G$, the inverse images of $x,y$ in $H_1$ commute if and only if
the inverse images in $H_2$ commute).
\item Two elements of $G$ are joined in the deep commuting graph if and
only if their inverse images in every central extension of $G$ commute.
\item The deep commuting graph of $G$ is invariant under the automorphism
group of $G$.
\end{enumerate}
\end{theorem}

\begin{proof}
(a) Let $H_1$ and $H_2$ be Schur covers of $G$, having projections
$\pi_i:H_i\to G$ with kernel $Z_i$ for $i=1,2$. Following Jones and Wiegold
\cite{jw}, we define
\[K=\{(h_1,h_2)\in H_1\times H_2:\pi_1(h_1)=\pi_2(h_2)\},\]
and let $\phi_i$ be the coordinate projection of $H_1\times H_2$ onto $H_i$
restricted to $K$, for $i=1,2$. Clearly $\pi_1\circ\phi_1=\pi_2\circ\phi_2$.
Moreover, $K$ is a central extension of $G$ with projection $\pi_1\circ\phi_1$,
with kernel $Z_1\times Z_2$.

It follows from our earlier remark that the commuting probability in
$K$ is equal to that in $H_1$ and $H_2$, so that two elements of $H_i$
commute if and only if their inverse images in $K$ commute. (The homomorphism
$\phi$ shows that $\kappa(K)\le\kappa(H_i)$, but $\kappa(H_i)$ is the
smallest commuting probability of any central extension of $G$.)

Let $\Gamma_1$ and $\Gamma_2$ be the deep commuting graphs on $G$ defined
by $H_1$ and $H_2$ respectively. If $x,y\in G$ with inverse images $a,b\in K$,
then we have
\begin{eqnarray*}
x\sim y\hbox{ in $\Gamma_1$} &\Leftrightarrow& [\phi_1(a),\phi_1(b)]=1
\hbox{ in $H_1$} \\
&\Leftrightarrow& [a,b]=1 \hbox{ in $K$} \\
&\Leftrightarrow& [\phi_2(a),\phi_2(b)]=1 \hbox{ in $H_2$} \\
&\Leftrightarrow& x\sim y \hbox{ in $\Gamma_2$}.
\end{eqnarray*}

\medskip

(b) Any central extension of $G$ is isoclinic to a stem extension, and so
inverse images of two elements of $G$ which commute in the Schur cover also
commute in the given extension.

\medskip

(c) This is now clear. \qed
\end{proof}

\paragraph{Remark} The automorphism group of $G$ permutes the Schur covers
of $G$. Without this theorem, it is clear that if $G$ has a unique Schur cover 
then the deep commuting graph is invariant under all group automorphisms. More
generally, Fry~\cite{fry} defined a class $\mathcal{L}$ of groups $G$ with the
property that $G$ has a Schur cover $H$ such that all automorphisms of $G$
lift to $H$ (in other words, $H$ is fixed by all automorphisms); these 
groups also clearly have the property that all group automorphisms preserve
the deep commuting graph.

\begin{theorem}
Let $G$ be a finite group. Then
\[E(\EPowGr(G))\subseteq E(\DComGr(G))\subseteq E(\ComGr(G)),\]
where $E(\Gamma)$ is the edge set of the graph $\Gamma$.
\end{theorem}

\begin{proof}
Let $H$ be a Schur cover of $G$ with kernel $Z$. Suppose that $x$ and $y$ are
joined in the enhanced power graph of $G$. Then there is an element $z\in G$
such that $x,y\in\langle z\rangle$. Let $a,b,c$ be the preimages of $x,y,z$
in $H$. Then $a,b\in\langle Z,c\rangle$, and this group is abelian since $Z$
is central. So $a$ and $b$ commute, and $x$ and $y$ are joined in the deep
commuting graph of $G$. 

The other inclusion is trivial since if two elements commute then so do
their images under a homomorphism.\qed
\end{proof}

Having these inclusions, it is natural to ask when equality holds. Aalipour
\emph{et al.}~\cite{aetal} characterised groups whose enhanced power graph and
commuting graph are equal: these are precisely the groups having no subgroup
$C_p\times C_p$ for any prime $p$ (so the Sylow subgroups are cyclic or
generalized quaternion). Our next job is to refine this.

\begin{theorem}
The deep commuting graph of the finite group $G$ is equal to its enhanced
power graph if and only if $G$ has the following property: Let $H$
be a Schur cover of $G$, with $H/Z=G$. Then for any subgroup $A$ of $G$, with
$B$ the corresponding subgroup of $H$ (so $Z\leqslant B$ and $B/Z=A$), if $B$
is abelian, then $A$ is cyclic.
\end{theorem}

\begin{proof}
Equality of these graphs requires that, if $x,y\in G$ and $a,b$ are inverse
images in $H$, if $a$ and $b$ commute, then $\langle x,y\rangle$ is cyclic.
This is certainly implied by the condition of the theorem, with
$A=\langle x,y\rangle$ and $B=\langle Z,a,b\rangle$. In the other direction,
suppose that the graphs are equal, and let $A\leqslant G$ be such that its
inverse image is abelian. Then for any $x,y\in A$, their inverse images in $B$
commute, and so $\langle x,y\rangle$ is cyclic. But if a finite group $A$ has
the property that any two elements generate a cyclic subgroup, then $A$ is
cyclic. \qed
\end{proof}

\section{The Bogomolov multiplier}

In order to explore equality of the commuting graph and deep commuting graph,
we need to discuss the Bogomolov multiplier of a finite group $G$. This arose
in connection with the work of Artin and Mumford on obstructions to Noether's
conjecture on the pure transcendence of the field of invariants; but we do
not need this background. We refer to \cite{Bog,Jez-Mor,Kun} for further
information.

Let $[x,y]=x^{-1}y^{-1}xy$ be the commutator. It is easy to check that in
every group $G$ the universal commuting identities $[xy,z]=[x^y,z^y][y,z]$,
$[x,yz]=[x,z][x^z,y^z]$, and $[x,x]=1$ hold. Based on this, Miller
\cite{Miller}  defined the \emph{nonabelian exterior square of $G$} as a group
generated by symbols $x\wedge y$, $x,y\in G$ subject to relations
\[(xy)\wedge z=(x^y\wedge z^y)(y\wedge z),\quad x\wedge (yz)=(x\wedge z)(x^z\wedge y^z),\quad x\wedge x=1.\]
It also was shown in \cite{Miller} that  $x\wedge y\mapsto [x,y]$ is a
surjective homomorphism from $G\wedge G$ onto $G$ whose  kernel is naturally
isomorphic to $M(G)$, the Schur multiplier of $G$. Denote
$M_0(G):=\langle x\wedge y\mid[x,y]=1\rangle$; then the group
$B_0(G):=M(G)/M_0(G)$ is known as the Bogomolov multiplier of $G$.
Observe that $B_0(G)$ is isomorphic to $M(G)$ if and only if they have the
same cardinality.

Let $H$ be a central extension of $G$, with $G=H/Z$. Clearly, if two elements
of $H$ commute, then their images under the projection from $H$ to $G$ also
commute. We say that the extension is \emph{commutation preserving} (for short,
CP) if the converse holds.

It was shown by Jezernik and Moravec in \cite[Theorem 4.2]{Jez-Mor} that there
exists a CP stem central  extension $K$ of $G$ with kernel isomorphic to
$B_0(G)$. Indeed, $B_0(G)$ is the largest possible kernel of a CP stem central
extension of $G$. If $|B_0(G)|=|M(G)|$, then $B_0(G)\simeq M(G)$ and clearly
$K$ is also a Schur cover of $G$. 

On the other hand, if $|B_0(G)| <|M(G)|$, then Schur cover of $G$ has larger
cardinality than $K$ and since $K$ was a stem central extension of maximal
cardinality which enjoys commuting lifting property, the Schur  cover does
not enjoy it. That is, the deep commuting graph has fewer edges than the
commuting graph of $G$.

So we have proved:

\begin{theorem}\label{th3}
For the finite group $G$, we have $\DComGr(G)=\ComGr(G)$ if and only if
the Bogomolov multiplier of $G$ is equal to the Schur multiplier.
\end{theorem}

\begin{cor}
Let $G$ be a finite group with $B_0(G)=1$. Then $\DComGr(G)=\ComGr(G)$ if
and only if $M(G)=1$.
\end{cor}

\paragraph{Remark} Kunyavski\u{\i}~\cite{Kun} proved a conjecture of
Bogomolov by showing that the Bogomolov multiplier of a finite simple group
is trivial. Hence the above corollary applies to finite simple groups.

\paragraph{Remark}
We observed that there is a map from quotients of $M(G)$ to spanning subgraphs
of the commuting graph of $G$, which is not one-to-one. We can now add to this
observation: the quotient $M(G)/M_0(G)=B_0(G)$ maps to the full commuting
graph. It may be that the map from quotients $M(G)/Z$ with $Z\le M_0(G)$ is
one-to-one; we have not investigated this.

\section{Examples}

\begin{example}
We begin with an example of a group whose deep commuting graph lies strictly
between the enhanced power graph and the commuting graph.

Take $n\geqslant8$, and let $G$ be the symmetric or alternating group of
degree $n$. Then it is known (see \cite{hh}) that the Schur multiplier of $G$
has order $2$. The alternating group has a unique Schur cover, while the
symmetric group has two; but all three share the property that an involution
in $G$ which is a product of $m$ transpositions lifts to an involution in the
cover if $n\equiv0\!\!\pmod{4}$, while it lifts to an element of order~$4$ if
$n\equiv2\!\!\pmod{4}$. Furthermore,
\begin{itemize}\itemsep0pt
\item $G$ has a noncyclic abelian subgroup isomorphic to $C_3\times C_3$,
whose lift is isomorphic to $C_2\times C_3\times C_3$, that is, abelian; so
$\DComGr(G)\ne\EPowGr(G)$.
\item The involutions
\[x=(1,2)(3,4)(5,6)(7,8) \hbox{ and } y=(1,3)(2,4)(5,6)(7,8)\]
are joined in $\ComGr(G)$ but not in $\DComGr(G)$, so these graphs are also
unequal.
\end{itemize}
\end{example}

\paragraph{Remark} We can formalise the argument in the first bullet point
above. If $G$ is a group and $p$ a prime such that $G$ has a subgroup isomorphic
to $C_p\times C_p$ but $p$ does not divide $|M(G)|$, then $\EPowGr(G)$ and
$\DComGr(G)$ are unequal.

\paragraph{Remark} Kunyavski\u{\i}'s theorem together with the preceding remark
gives other examples of simple groups with all three graphs distinct, such as
the Mathieu group $M_{12}$.

\begin{example} There exists a group $G$  of order $|G|=2^6$ with
nontrivial Schur multiplier such that its  deep commuting graph coincides with
its commuting graph. To see this, let $G=C_8\rtimes Q_8$ be a split extension
of quaternion group by a cyclic group of order eight (this is SmallGroup(64,182) in the GAP~\cite{GAP} library).   Using GAP we compute
that  its Schur multiplier equals $M(G)=C_2$. Take any two commuting elements
$x,y\in G$. If $\langle x,y\rangle$ cyclic, then by (a) of
Theorem~\ref{th3} $x,y$ have a commuting lift to $H$, the Schur cover
of $G$. Assume $\langle x,y\rangle$ is nocyclic. Again, with GAP computations
one sees that, up to conjugacy, the group $G$ has exactly $11$ noncyclic
abelian subgroups and each lifts to an abelian subgroup of $H$. 

In view of Theorem~\ref{th3}  this group $G$ is an example of a group where
Bogomolov and Schur multiplier are isomorphic and nontrivial.
\end{example}

\begin{example}
The dihedral $2$-group,
\[H=D_{2^{n+1}}=\langle a,b\mid a^{2^{n}}=b^2=1, a^b=a^{-1}\rangle\]
is a Schur cover of $G=D_{2^{n}}$ for $n\ge 3$. 

The centre $Z=Z(H)$ of $H$ is generated by $a^{2^{n-1}}$ so the commutator,
$[a^{2^{n-2}},b]=a^{2^{n-1}}\in Z$. That is, $a^{2^{n-2}}$ and $b$ do not
commute in $H$ but they do commute in $G\simeq H/Z$.
This shows that the cover is not CP. So $\DComGr(G)$ is strictly contained in
$\ComGr(G)$. 

Howevever, $\EPowGr(G)=\DComGr(G)$ because if $x,y \in H=D_{2^{n+1}}$
commute and the group $\langle x,y\rangle$ is not cyclic, then  $xy^{-1}$
belongs to the center of $ D_{2^{n+1}}$ so in the quotient group
$H/Z\simeq D_{2^{n}}$ we have $ xZ =yZ$, i.e., $\langle xZ,yZ\rangle$ is a 
cyclic group in $G$.
\end{example}

\section{Further properties}

Unlike the enhanced power graph and commuting graph, the
deep commuting graph of $G$ does not have the property that the induced
subgraph on a subgroup $H$ is the deep commuting graph of $H$. The above
groups provide examples. The deep commuting graph of $C_3\times C_3$ consists
of four triangles with a common vertex; but the induced subgraph of the deep
commuting graph of $S_n$ on a $C_3\times C_3$ subgroup is the complete graph
on nine vertices.

\medskip

We have not considered graph-theoretic properties of the deep commuting graph
such as connectedness or chromatic number.

\paragraph{Acknowledgement} The authors are grateful to Sean Eberhard, Saul
Freedman and Michael Giudici for helpful comments.

\end{document}